\documentclass[letterpaper, 10 pt, conference]{ieeeconf}  

\IEEEoverridecommandlockouts                              
\usepackage{amsmath} 
\usepackage{amssymb}  

\usepackage{xcolor}
\usepackage{xparse}

\usepackage{mathtools,bm}
\usepackage{booktabs}
\usepackage{diagbox}
\usepackage{multirow}
\usepackage{balance}
\newtheorem{theorem}{Theorem}
\newtheorem{assumption}{Assumption}

\newtheorem{lemma}{Lemma}

\newtheorem{remark}{Remark}

\newcommand{\lucacolor}{black}
\NewDocumentCommand{\luca}{+m}{{\color{\lucacolor}#1}}

\usepackage{graphicx}
\usepackage{arydshln}
\usepackage{cite}
\makeatletter
\let\NAT@parse\undefined
\makeatother
\usepackage[colorlinks=true, linkcolor=blue, citecolor=blue, urlcolor=blue]{hyperref}

\title{\LARGE \bf
Learning  Over-Relaxation Policies for ADMM\\ with Convergence Guarantees

}

\author{Junan Lin, Paul J. Goulart, Luca Furieri
\thanks{The authors are with the Department of Engineering Science, University of Oxford, OX1 3PJ Oxford, UK. Emails: \texttt{\{junan.lin, paul.goulart, luca.furieri\}@eng.ox.ac.uk}}}

\begin{document}

\maketitle
\thispagestyle{empty}
\pagestyle{empty}

\begin{abstract}
The Alternating Direction Method of Multipliers (ADMM) is a widely used method for structured convex optimization, and its practical performance depends strongly on the choice of penalty and relaxation parameters. Motivated by settings such as Model Predictive Control (MPC), where one repeatedly solves related optimization problems with fixed structure and changing parameter values, we propose learning online updates of the relaxation parameter to improve average performance on problem classes of interest, while guaranteeing that asymptotic convergence is not compromised for the worst-case realization of such problems. This choice is computationally attractive in the Operator Splitting Quadratic Program (OSQP)-like architectures, since adapting relaxation does not trigger the matrix refactorizations associated with penalty updates. We establish convergence guarantees for ADMM with time-varying penalty and relaxation parameters under mild assumptions, and show on benchmark quadratic programs that the resulting learned policies improve both iteration count and wall-clock time on average over baseline OSQP.
\end{abstract}

\section{Introduction}

The Alternating Direction Method of Multipliers (ADMM) is one of the most widely used first-order methods for structured convex optimization. It combines inexpensive iterations, scalability to large problem dimensions, and broad applicability across signal processing, machine learning, control, and embedded optimization \cite{neal2011distributed,nishihara2015general}. For problems with favorable structure, ADMM often provides an attractive compromise between per-iteration cost and practical accuracy, and has therefore become a standard computational primitive in many repeated-solve applications \cite{neal2011distributed,stellato2020osqp}.

Consider the convex optimization problem
\begin{equation}
\begin{aligned}
\label{eq:opt_problem}
\min_{x,z}\quad & f(x)+g(z)\\
\text{s.t.}\quad & Ax + Bz = c,
\end{aligned}
\end{equation}
where $x \in \mathbb{R}^{n}$, $z \in \mathbb{R}^{m}$, $A \in \mathbb{R}^{p \times n}$, $B \in \mathbb{R}^{p \times m}$, $c \in \mathbb{R}^{p}$, and $f$ and $g$ are closed proper convex functions. In its basic form, ADMM generates iterates
\begin{subequations}
\label{eq:admm_intro}
\begin{align}
x_{k+1} &\in \arg\min_x f(x) + \frac{\rho_k}{2}\|Ax \!+ \!Bz_k \!-\! c \!+\! u_k\|_2^2, \label{eq:admm_intro_x}\\
z_{k+1} &\in \arg\min_z g(z) + \frac{\rho_k}{2}\|Ax_{k+1} \!+\! Bz \!-\! c \!+\! u_k\|_2^2, \label{eq:admm_intro_z}\\
u_{k+1} &= u_k + Ax_{k+1}+Bz_{k+1}-c, \label{eq:admm_intro_u}
\end{align}
where $\rho_k>0$ is the \emph{penalty parameter} and $u$ is the scaled dual variable \cite{neal2011distributed}. A standard extension is relaxed ADMM, in which 
the following term 
\begin{equation}
    \tilde{x}_{k+1} = \alpha_k Ax_{k+1} + (1-\alpha_k)(c-Bz_k), \label{eq:admm_intro_relax}
\end{equation}
\end{subequations}
is computed immediately after \eqref{eq:admm_intro_x}, and $\tilde{x}_{k+1}$ is used in place of $Ax_{k+1}$ in \eqref{eq:admm_intro_z} and \eqref{eq:admm_intro_u}.  The parameter  $\alpha_k\in(0,2)$ is the \emph{relaxation parameter}, and the combined choice of $\rho_k$ and $\alpha_k$ can have a major impact on practical convergence speed \cite{nishihara2015general,deng2016global}. 

The sensitivity of the algorithm performance to these hyperparameters is particularly important in applications where one repeatedly solves a family of related optimization problems with similar structure. A standard example is Model Predictive Control (MPC), where the same optimization problem is solved online for changing initial conditions. In such settings, the relevant performance criterion is often not the average behavior over a broad class of problems, but rather the performance on the narrower distribution of instances induced by the application. This motivates adapting the solver to the specific class of problems of interest.

This observation has motivated adaptive ADMM schemes. Adaptive ADMM (AADMM) updates the penalty parameter using spectral rules derived from local curvature estimates \cite{xu2017aadmm}. Adaptive Relaxed ADMM (ARADMM) extends this idea to jointly adapt scalar penalty and relaxation parameters \cite{xu2017aradmm}. More recently, superADMM proposed constraint-wise penalty adaptation for quadratic programs \cite{verheijen2025superadmm}. These methods show that adaptive hyperparameter updates can yield substantial performance gains. 

Most of the effort in this direction has focused on adaptation of the penalty parameter $\rho$, which brings computational disadvantages. Most notably, the Operator Splitting Quadratic Program solver (OSQP) \cite{stellato2020osqp} is built around repeated solution of a linear system that is parametric in $\rho$. Updating $\rho$ modifies this linear system and can trigger a costly refactorization, which limits how frequently one should adapt it in practice. In contrast, updating the relaxation parameter $\alpha$ does not modify this factorization and can be done online at negligible additional cost. This makes $\alpha$ an attractive target for adaptation, especially in repeated-solve settings where wall-clock time matters as much as iteration~count.




More broadly, hyperparameter adaptation is aligned with the success of recent approaches based on learning-to-optimize (L2O). L2O seeks to improve algorithmic performance on a distribution of problems of interest by exploiting data from representative instances \cite{andrychowicz2016learning,chen2022learning,sambharya2024learning,sambharya2025data,sambharya2024learningwarm,martin2024learning,martin2025learning,martin2026learning}.
In the context of ADMM and related fixed-point methods, learning can substantially improve performance on optimization problem families arising in control \cite{sambharya2025data,martin2025learning}. At the same time, learned modifications must remain compatible with solver structure and convergence guarantees. In particular, learning high-dimensional perturbations of the iteration map can be difficult to train and may not generalize reliably \cite{martin2024learning,martin2025learning,martin2026learning}, whereas learning a low-dimensional solver parameter as per the RLQP approach of \cite{ichnowski2021acceleratingrlqp} can offer lighter-weight computations and higher transferability to new problem instances.

Motivated by these considerations, in this paper we investigate online adaptation of the ADMM relaxation parameter. We first derive convergence guarantees for ADMM with time-varying hyperparameter vectors (different parameters  $\rho_k^i$ and $\alpha_k^i$ for each row of $A$ in \eqref{eq:admm_intro_x}-\eqref{eq:admm_intro_relax}) for a general class of convex optimization problems under mild assumptions. We then demonstrate the practical potential of this idea on quadratic programs solved in OSQP-style form, where policies trained on representative problem instances map solver-state features to constraint-wise values of $\alpha_k^i$. In this implementation, adapting the over-relaxation parameters incurs negligible computational overhead and can be combined naturally with existing $\rho$-adaptation strategies.



\textbf{Contributions. } Our contribution is twofold. First, we prove asymptotic convergence of ADMM with time-varying diagonal penalty and relaxation matrices, including the case where both parameters are adapted independently across constraints. Second, we develop a learning-to-optimize pipeline for online adaptation of the ADMM relaxation parameter in OSQP. The resulting learned policy can improve iteration count and wall-clock time on representative problem families, while preserving the solver structure and avoiding the factorization overhead caused by frequent $\rho$ updates. Our experiments show that learned adaptation of ADMM relaxation parameters can outperform baseline OSQP across several benchmark families. In particular, noticeable wall-clock time reduction can be obtained even with fixed $\rho$, since adapting $\alpha$ improves convergence behavior without changing the linear system factorization.

\section{Main Results}\label{sec:convergence_theory}

We first present the theoretical result that supports the adaptive scheme proposed in this paper. Our goal is to adapt the ADMM relaxation parameter online while preserving convergence guarantees, even in the presence of updates for penalty parameter. Existing results do not cover the setting in which both parameters vary across iterations and constraints, while the result established in this section shows that per-constraint adaptation of $\alpha_k$ in \eqref{eq:admm_intro} is compatible with existing adaptive penalty rules such as the one in OSQP \cite{stellato2020osqp}. Asymptotic convergence is then guaranteed provided that relaxation and penalty parameter sequences satisfy suitable summable-change conditions.

We make the following assumptions throughout:

\begin{assumption}
\label{ass:well_defined}
The functions $f:\mathbb{R}^n \to \mathbb{R}\cup\{+\infty\}$ and $g:\mathbb{R}^m \to \mathbb{R}\cup\{+\infty\}$ are closed, proper, convex functions. 
\end{assumption}

\begin{assumption}
\label{ass:saddle}
The Lagrangian function
\begin{equation*}
L_0(x, z, \lambda) = f(x) + g(z) + \lambda^\top (Ax + Bz - c)
\end{equation*}
has a saddle point $(x^\star,z^\star,\lambda^\star)$, that is,
\begin{equation*}
L_0(x^\star, z^\star, \lambda) \le L_0(x^\star, z^\star, \lambda^\star) \le L_0(x, z, \lambda^\star)
\end{equation*}
for all $(x,z,\lambda)$.
\end{assumption}

Under Assumption~\ref{ass:saddle}, the pair $(x^\star,z^\star)$ is an optimal solution of \eqref{eq:opt_problem}, and the optimal value is given by
\begin{equation*}
p^\star := f(x^\star)+g(z^\star).
\end{equation*}

To state our main result, it is convenient to write the relaxed ADMM iterations using matrix-valued penalty and relaxation parameters. Specifically, we replace the scalar parameters $\rho_{k}$ and $\alpha_{k}$ by diagonal matrices $R_k$ and $\Gamma_k$, respectively,  whose diagonal entries represent parameters assigned separately to each constraint. The resulting relaxed ADMM iteration is
\begin{subequations}\label{eq:matrix-admm}
\begin{align}
\!\! x_{k+1} &\in \arg\min_x \; f(x) \!+\! \frac12 \bigl\|Ax \!+\! Bz_k \!-\! c \!+\! R_k^{-1}\lambda_k\bigr\|_{R_k}^2\!,
\label{eq:matrix-x-update}\\
\!\!\tilde x_{k+1} &:= \Gamma_k A x_{k+1} - (I-\Gamma_k)(Bz_k-c),
\label{eq:matrix-x-relaxed}\\
\!\!z_{k+1} &\in \arg\min_z \; g(z) \!+\! \frac12 \bigl\|\tilde x_{k+1} \!+\! Bz \!-\! c \!+\! R_k^{-1}\lambda_k\bigr\|_{R_k}^2\!,
\label{eq:matrix-z-update}\\
\!\!\lambda_{k+1} &:= \lambda_k + R_k\bigl(\tilde x_{k+1}+Bz_{k+1}-c\bigr).
\label{eq:matrix-lam-update}
\end{align}
\end{subequations}

When $R_k=\rho_k I$ and $\Gamma_k=\alpha_k I$, \eqref{eq:matrix-admm} reduces to the standard scalar-parameter relaxed ADMM iteration. Based on these iterations, we state an additional assumption on well-posedness of \eqref{eq:matrix-x-update} and \eqref{eq:matrix-z-update}.

\begin{assumption}
\label{ass:subproblems}
For every $k$, the ADMM subproblems \eqref{eq:matrix-x-update} and \eqref{eq:matrix-z-update} admit at least one minimizer in $\mathbb{R}^n$ and $\mathbb{R}^m$, respectively.
\end{assumption}

Assumption~\ref{ass:subproblems} is fairly mild, and it is satisfied whenever the sublevel sets of the cost functionals in \eqref{eq:matrix-x-update} and \eqref{eq:matrix-z-update} are bounded. 
Our main result is the following:

\begin{theorem}\label{thm:main}
Consider problem \eqref{eq:opt_problem} under Assumptions~\ref{ass:well_defined}--\ref{ass:subproblems}, and let $\{x_k,z_k,\lambda_k\}$ be generated by the matrix-valued relaxed ADMM iteration \eqref{eq:matrix-admm}. Assume that $R_k$ and $\Gamma_k$ are positive definite diagonal matrices satisfying
\begin{equation*}
\rho_{\min}I \preceq R_k \preceq \rho_{\max}I,
\qquad
\alpha_{\min}I \preceq \Gamma_k \preceq \alpha_{\max}I
\end{equation*}
for all $k$, where $0<\rho_{\min}\le \rho_{\max}<\infty$ and $0<\alpha_{\min}\le \alpha_{\max}<2$. Further assume that there exist nonnegative sequences $\{\underline{\eta}_k\}$, $\{\overline{\eta}_k\}$, $\{\underline{\xi}_k\}$, and $\{\overline{\xi}_k\}$, such that
$
\underline{\eta}_k < 1$ and 
$\underline{\xi}_k < 1
$
for all $k$, and
\begin{align}
(1-\underline{\eta}_k)R_k &\preceq R_{k+1} \preceq (1+\overline{\eta}_k)R_k, \label{eq:Rk_ineq}
\\
(1-\underline{\xi}_k)\Gamma_k &\preceq \Gamma_{k+1} \preceq (1+\overline{\xi}_k)\Gamma_k \label{eq:Gammak_ineq}
\end{align}
for all $k$, and that, with
\begin{equation*}
\theta_k^R := \max\Bigl\{\frac{\underline{\eta}_k}{1-\underline{\eta}_k},\,\overline{\eta}_k\Bigr\},
\qquad
\theta_k^\Gamma := \max\Bigl\{\frac{\underline{\xi}_k}{1-\underline{\xi}_k},\,\overline{\xi}_k\Bigr\},
\end{equation*}
one has $\sum_{k=0}^\infty (\theta_k^R+\theta_k^\Gamma)<\infty$. Then:
\begin{enumerate}
\item the primal residual $r_{k+1}:=Ax_{k+1}+Bz_{k+1}-c$ satisfies $\lim_{k\rightarrow \infty}r_{k+1}= 0$;
\item the dual residual $s_{k+1}:=A^\top R_k B(z_{k+1}-z_k)$ satisfies $\lim_{k\rightarrow \infty}s_{k+1}= 0$;
\item the objective value satisfies $\lim_{k\rightarrow \infty}p_k= p^\star$\,,
\end{enumerate}
where $p_k := f(x_k)+g(z_k)$ and $p^\star := f(x^\star)+g(z^\star)$.
\end{theorem}


Before proving Theorem~\ref{thm:main} we require several intermediate results. Section~\ref{sec:A} rewrites ADMM as a relaxed Douglas--Rachford iteration with a perturbation induced by the metric update; Section~\ref{sec:B} establishes a one-step descent estimate for the relaxed map; 
and Section~\ref{sec:C} shows that both the relaxed Douglas--Rachford 
step and the metric-induced perturbation vanish asymptotically.
These results are then combined in Section~\ref{sec:D} to prove Theorem~\ref{thm:main}.

\subsection{ADMM as relaxed Douglas--Rachford splitting}
\label{sec:A}
We can rewrite the matrix-valued relaxed ADMM iteration \eqref{eq:matrix-admm} as a relaxed Douglas--Rachford splitting applied to the dual problem. This representation isolates the effect of the relaxation matrix $\Gamma_k$ from the effect of the time variation of the penalty matrix $R_k$, and will be the starting point for the descent estimate in the next subsection. Let $\sigma_k := c - Bz_k$. We introduce the state variables
\begin{equation}
y_{k+1} := \lambda_{k+1} + R_{k+1}\sigma_{k+1},
\quad
\tilde y_{k+1} := \lambda_{k+1} + R_k\sigma_{k+1}.
\label{eq:ytildey-def}
\end{equation}

The vector $\tilde y_{k+1}$ is the state produced by one relaxed Douglas--Rachford step in the current metric $R_k$ \cite{giselsson2016linear}, while $y_{k+1}$ is the actual successor state after the metric is updated from $R_k$ to $R_{k+1}$.

Following \cite{eckstein1992douglas,lorenz2025degenerate}, define the convex dual functions
\begin{equation*}
\tilde f(\lambda) := f^\star(-A^\top \lambda),
\qquad
\tilde g(\lambda) := g^\star(-B^\top \lambda) + c^\top \lambda,
\end{equation*}
and the corresponding maximal monotone operators
\begin{equation*}
\mathcal A := \partial \tilde f,
\qquad
\mathcal B := \partial \tilde g,
\end{equation*}
where $\partial$ denotes the subdifferential. The dual optimality condition can then be written as $0 \in \mathcal A(\lambda) + \mathcal B(\lambda)$. For each  $R_k\succ 0$, define the associated resolvents
\begin{equation*}
J_{R_k\mathcal A} := (I+R_k\mathcal A)^{-1},
\qquad
J_{R_k\mathcal B} := (I+R_k\mathcal B)^{-1},
\end{equation*}
and the Douglas--Rachford operator
\begin{equation}
T_k(y):=
J_{R_k\mathcal A}\bigl(2J_{R_k\mathcal B}(y)-y\bigr)
+
y-J_{R_k\mathcal B}(y).
\label{eq:Tk-def}
\end{equation}

When $\Gamma_k = I$, one iteration of \eqref{eq:matrix-admm} is equivalent to applying the Douglas--Rachford operator $T_k$ once to the dual inclusion above \cite{eckstein1992douglas,lorenz2025degenerate}. When $\Gamma_k$ is not necessarily the identity, a direct calculation using \eqref{eq:matrix-x-relaxed}--\eqref{eq:matrix-lam-update}, together with the commutativity of the diagonal matrices $\Gamma_k$ and $R_k$, yields
\begin{equation}
\tilde y_{k+1}
=
(I-\Gamma_k)y_k + \Gamma_k T_k(y_k)
=: S_k(y_k),
\label{eq:Sk-def}
\end{equation}
where $S_k$ denotes the relaxed Douglas--Rachford operator. We then have 
\begin{gather}
\tilde{y}_{k+1} - y_k = H_k^{-1} \, e_{k+1}, \label{eq:ytilde-diff}\\
y_{k+1} - \tilde{y}_{k+1}= E_k\,\sigma_{k+1}, \label{eq:actual-state}
\end{gather}
 where $H_k := \Gamma_k^{-1} R_k^{-1}$, $e_{k+1} := Ax_{k+1}+Bz_k-c$, and  $E_k := R_{k+1}-R_k$.

Equations \eqref{eq:Sk-def} and \eqref{eq:actual-state} show that the transition from $y_k$ to $y_{k+1}$ can be decomposed into two stages, as in \cite{lorenz2025degenerate}. First, the current state $y_k$ is mapped through the relaxed Douglas--Rachford operator $S_k$, producing the intermediate point $\tilde y_{k+1}$. Second, the penalty matrix is updated from $R_k$ to $R_{k+1}$, which introduces the perturbation term $E_k\sigma_{k+1}$ and yields the actual next state $y_{k+1}$.

\subsection{One-step descent for the relaxed DRS map}
\label{sec:B}

We now derive the one-step descent estimate for the relaxed Douglas--Rachford map $S_k$. The estimate shows that one application of $S_k$ decreases the distance to a reference fixed point by an amount controlled by the step $\tilde y_{k+1}-y_k$. 

Since $\mathcal{A}$ and $\mathcal{B}$ are maximal monotone, it follows from \cite[Lemma~2.2]{lorenz2025degenerate} that $J_{R_k\mathcal{A}}$, $J_{R_k\mathcal{B}}$, and $T_k$ are firmly nonexpansive in the norm $\|\cdot\|_{R_k^{-1}}$. Since the saddle point need not be unique, we fix an arbitrary saddle point $(x^\star,z^\star,\lambda^\star)$ of $L_0$ and define
\begin{equation*}
H_k := \Gamma_k^{-1}R_k^{-1}, ~~ \sigma^\star := c-Bz^\star = Ax^\star, ~~ y_{R_k}^\star := \lambda^\star + R_k \sigma^\star.
\end{equation*}

By the saddle-point condition, it holds that $y_{R_k}^\star \in \mathrm{Fix}(T_k)$, where $\mathrm{Fix}(T)$ denotes the set of fixed points for an operator $T$. By direct calculation through \eqref{eq:Sk-def} we can derive $\mathrm{Fix}(S_k) = \mathrm{Fix} (T_k)$, and therefore $y_{R_k}^\star \in \mathrm{Fix}(S_k)$.

\begin{remark}
The convergence claim in Theorem~\ref{thm:main} is stated only in terms of the primal residual, dual residual, and objective value, since $(x^\star,z^\star,\lambda^\star)$ need not be unique. In particular, we do not claim that $y_k$, $\tilde y_k$, or $y_{R_k}^\star$ converges to a specific value.
\end{remark}

\begin{lemma}\label{lem:averaged}
For each $k$, it holds that
\begin{equation}\label{eq:one-step}
\|\tilde{y}_{k+1}-y_{R_k}^\star\|_{H_k}^2 \le \|y_k-y_{R_k}^\star\|_{H_k}^2 - \kappa\,\|\tilde{y}_{k+1}-y_k\|_{H_k}^2
\end{equation}
with $\kappa := (2/\alpha_{\max} - 1) > 0$.
\end{lemma}
\vspace{8pt}

\begin{proof}
The proof adapts the one-step descent argument from \cite[Lemma~2.2 and Lemma~2.5]{lorenz2025degenerate} to the relaxed setting with per-constraint relaxation parameters.
Write $Q_k = R_k^{-1}$. Since $R_k$ and $\Gamma_k$ are diagonal, $H_k\Gamma_k = Q_k$. Let $F_k := I-T_k$. Since $T_k$ is firmly nonexpansive in $\|\cdot\|_{Q_k}$, so is $F_k$, i.e.\
\begin{equation*}
\|F_k(x)-F_k(y)\|_{Q_k}^2 \le \langle x-y,\,F_k(x)-F_k(y)\rangle_{Q_k}
\end{equation*}
for all $x,y$. Let $u := x-y$ and $d := F_k(x)-F_k(y)$. Since $S_k = I-\Gamma_kF_k$,
\begin{align*}
\|S_k(x)-S_k(y)\|_{H_k}^2
&= \|u-\Gamma_k d\|_{H_k}^2 \\
&= \|u\|_{H_k}^2 + \|\Gamma_k d\|_{H_k}^2 - 2\langle u,\Gamma_k d\rangle_{H_k}.
\end{align*}
Using $H_k\Gamma_k = Q_k$ and $\Gamma_k\preceq \alpha_{\max}I$, we obtain
\begin{equation*}
\|\Gamma_k d\|_{H_k}^2 = d^\top Q_k \Gamma_k d \le \alpha_{\max}\|d\|_{Q_k}^2.
\end{equation*}
Using firm nonexpansiveness of $F_k$, we also have
\begin{equation*}
\langle u,\Gamma_k d\rangle_{H_k} = \langle u,d\rangle_{Q_k}
\ge \|d\|_{Q_k}^2
\ge \alpha_{\max}^{-1}\|\Gamma_k d\|_{H_k}^2.
\end{equation*}
Substituting these bounds yields
\begin{align*}
\|S_k(x)-S_k(y)\|_{H_k}^2
\le {} & \|x-y\|_{H_k}^2 \\
& {} - \kappa\,\|(I-S_k)(x)-(I-S_k)(y)\|_{H_k}^2,
\end{align*}
with $\kappa = 2/\alpha_{\max}-1$. Applying this inequality with $x=y_k$ and $y=y_{R_k}^\star\in \mathrm{Fix}(S_k)$ gives \eqref{eq:one-step}. 
\end{proof}

Lemma~\ref{lem:averaged} gives a one-step descent estimate for the relaxed Douglas--Rachford map with the metric and reference point frozen at iteration $k$. The remaining issue is that both quantities vary with $k$ through $R_k$ and $\Gamma_k$.

\subsection{Asymptotic regularity and vanishing perturbation}
\label{sec:C}
We next show that both the iterate difference \eqref{eq:ytilde-diff} and the perturbation \eqref{eq:actual-state} vanish asymptotically.


\begin{lemma}\label{lem:quasi-fejer}
The iteration \eqref{eq:ytilde-diff} is asymptotically regular and the perturbation \eqref{eq:actual-state} is vanishing, in the sense that
\begin{equation}\label{eq:ytilde-vanish}
\|\tilde{y}_{k+1}-y_k\|_{H_k}\to 0, \quad
y_{k+1}-\tilde{y}_{k+1} = E_k\sigma_{k+1}\to 0.
\end{equation}
\end{lemma}
\vspace{8pt}


\begin{proof}
Define $a_k:=\|y_k-y_{R_k}^\star\|_{H_k}^2$ and $b_k:=\kappa\|\tilde{y}_{k+1}-y_k\|_{H_k}^2$. 
From~\eqref{eq:actual-state}, we obtain 
\begin{equation}\label{eq:y_ystar}
y_{k+1}-y_{R_{k+1}}^\star = (\tilde{y}_{k+1}-y_{R_k}^\star)+E_k(\sigma_{k+1}-\sigma^\star).
\end{equation}

From the $z$-optimality condition of~\eqref{eq:matrix-z-update} it holds that $\sigma_{k+1}\in\mathcal{B}(\lambda_{k+1})$, so $\tilde{y}_{k+1} = \lambda_{k+1}+R_k\sigma_{k+1}$ gives $\lambda_{k+1}=J_{R_k\mathcal{B}}(\tilde{y}_{k+1})$. Similarly, from the saddle-point conditions, $\sigma^\star\in\mathcal{B}(\lambda^\star)$ gives $\lambda^\star = J_{R_k\mathcal{B}}(y_{R_k}^\star)$. 

Since $J_{R_k\mathcal{B}}$ is firmly nonexpansive in $\|\cdot\|_{R_k^{-1}}$, so is $(I-J_{R_k\mathcal{B}})$. Noting $(I-J_{R_k\mathcal{B}})(\tilde{y}_{k+1}) = R_k\sigma_{k+1}$ and $(I-J_{R_k\mathcal{B}})(y_{R_k}^\star)=R_k\sigma^\star$, we get
\begin{equation}\label{eq:sigma_ineq}
\begin{aligned}
\|\sigma_{k+1}-\sigma^\star\|_{R_k} &= \|R_k(\sigma_{k+1}-\sigma^\star)\|_{R_k^{-1}}\\ 
&\le \sqrt{\alpha_{\max}}\,\|\tilde{y}_{k+1}-y_{R_k}^\star\|_{H_k},
\end{aligned}
\end{equation}
where the inequality uses firm nonexpansiveness of $(I-J_{R_k\mathcal{B}})$ and $R_k^{-1} = \Gamma_k H_k \preceq \alpha_{\max}\, H_k$.


Since $E_k$ satisfies $-\theta_k^R R_k\preceq E_k\preceq \theta_k^R R_k$, with \eqref{eq:y_ystar} and \eqref{eq:sigma_ineq} we derive
\begin{equation}\label{eq:perturbation}
\begin{aligned}
\|y_{k+1}-y_{R_{k+1}}^\star\|_{H_k} &\le  \| \tilde{y}_{k+1}-y_{R_k}^\star\|_{H_k} + \| E_k(\sigma_{k+1}-\sigma^\star) \|_{H_k}\\
&\le (1+c_1\theta_k^R)\,\|\tilde{y}_{k+1}-y_{R_k}^\star\|_{H_k},
\end{aligned}
\end{equation}
with $c_1=\sqrt{\alpha_{\max}/\alpha_{\min}}$.
%
From \eqref{eq:Rk_ineq} and \eqref{eq:Gammak_ineq}, we have
\begin{gather*}
R_{k+1}^{-1}\preceq \frac{1}{1-\underline{\eta}_k}R_k^{-1}
\preceq (1+\theta_k^R)R_k^{-1}, \\
\Gamma_{k+1}^{-1}\preceq \frac{1}{1-\underline{\xi}_k}\Gamma_k^{-1}
\preceq (1+\theta_k^\Gamma)\Gamma_k^{-1}.
\end{gather*}
Since $R_k$ and $\Gamma_k$ are diagonal, all of these matrices commute. Hence
\begin{equation}\label{eq:metric_drift}
H_{k+1}\preceq (1+\delta_k)H_k,
\end{equation}
where $\delta_k:=\theta_k^R+\theta_k^\Gamma+\theta_k^R\theta_k^\Gamma$ is nonnegative and summable.

Combining Lemma \ref{lem:averaged},~\eqref{eq:perturbation}, and \eqref{eq:metric_drift}, we can prove:
\[
a_{k+1} \le (1+\delta_k)(1+\omega_k)(a_k-b_k) \le (1+\varepsilon_k)a_k - b_k,
\]
where $\omega_k := 2c_1\theta_k^R + c_1^2(\theta_k^R)^2$ and $\varepsilon_k := \omega_k+\delta_k+\omega_k\delta_k$ are nonnegative and summable.

Define $\beta_0 := 1$, $\beta_{k+1} := \beta_k/(1+\varepsilon_k)$, and $w_k := \beta_k a_k$. Then $w_{k+1}\le w_k - \beta_{k+1}b_k$, so $\{w_k\}$ is nonincreasing and bounded, where $\{ w_k\}$ denotes the sequence of values $w_t$ for each $t \in \mathbb{N}$.


Since $\{ \varepsilon_k \}$ is nonnegative and summable, $\{\beta_k\}$ is lower bounded by a positive number, and therefore $\sum_k b_k<\infty$ and $b_k \to 0$. From the definition of $b_k$, we get $\| \tilde{y}_{k+1}-y_k \|_{H_k} \to 0$. This proves the first claim of the lemma.

Because $\{\beta_k\}$ is lower bounded by a positive number and $\{w_k\}$ is bounded, $\{ a_k \}$ is bounded. Note that we cannot ensure convergence of $\{ a_k \}$. Then Lemma \ref{lem:averaged} gives $\|\tilde{y}_{k+1}-y_{R_k}^\star\|_{H_k}^2\le a_k$, so $\{\tilde{y}_{k+1}\}$ is bounded. Together with \eqref{eq:sigma_ineq}, we get $\{\sigma_{k+1}\}$ is bounded. From $-\theta_k^R R_k \preceq E_k \preceq \theta_k^R R_k$ with $R_k\preceq \rho_{\max}I$, we have
\[
\|E_k\|\le \theta_k^R \|R_k\|\le \rho_{\max}\theta_k^R \to 0.
\]

Since $\{\sigma_{k+1}\}$ is bounded and $y_{k+1}-\tilde{y}_{k+1} = E_k\sigma_{k+1}$, it follows that
\[
y_{k+1}-\tilde y_{k+1}\to 0.
\]

This proves the second claim. 
\end{proof}

\subsection{Proof of Theorem~\ref{thm:main}}
\label{sec:D}
We now combine the reformulation from Section~\ref{sec:A}, the descent estimate from Section~\ref{sec:B}, and the asymptotic regularity and vanishing perturbation results from Section~\ref{sec:C} to prove our main result:


\begin{proof}[Theorem~\ref{thm:main}]
The proof is adapted from \cite[Appendix A]{neal2011distributed}.
We first prove convergence of the primal and dual residuals, and then of the objective value. Define $d_k:=\sigma_k-\sigma_{k+1}$,
so that
\begin{equation*}
r_{k+1}=Ax_{k+1}+Bz_{k+1}-c=e_{k+1}+d_k.
\end{equation*}

From \eqref{eq:ytilde-diff} and Lemma~\ref{lem:quasi-fejer}, we have $\|e_{k+1}\|_{\Gamma_k R_k}\to 0$. Since $\Gamma_k R_k\succeq \alpha_{\min}\rho_{\min}I$, it follows that $e_{k+1}\to 0$.

We next show that $d_k\to 0$. The optimality conditions of \eqref{eq:matrix-z-update} give $-B^\top\lambda_{k+1}\in\partial g(z_{k+1})$ and $-B^\top\lambda_k\in\partial g(z_k)$. By monotonicity of $\partial g$,
\begin{equation*}
\langle z_{k+1}-z_k,\,-B^\top(\lambda_{k+1}-\lambda_k)\rangle \ge 0,
\end{equation*}
that is, $\langle d_k,\lambda_{k+1}-\lambda_k\rangle\le 0$. Using \eqref{eq:matrix-lam-update}, namely $\lambda_{k+1}-\lambda_k=R_k(\Gamma_ke_{k+1}+d_k)$, we obtain
\begin{align*}
\|d_k\|_{R_k}^2
&\le -\langle d_k,\Gamma_ke_{k+1}\rangle_{R_k}
\le \|d_k\|_{R_k}\,\|\Gamma_ke_{k+1}\|_{R_k} \\
&\le \alpha_{\max}\|d_k\|_{R_k}\|e_{k+1}\|_{R_k},
\end{align*}
where the last inequality uses $\Gamma_k\preceq \alpha_{\max}I$. Hence $\|d_k\|_{R_k}\le \alpha_{\max}\|e_{k+1}\|_{R_k}\to 0$, and since $R_k\succeq \rho_{\min}I$, we conclude that $d_k\to 0$. Therefore $r_{k+1}=e_{k+1}+d_k\to 0$, and the dual residual $s_{k+1}=A^\top R_k d_k\to 0$ because $\{R_k\}$ is uniformly bounded.

We now prove convergence of the objective value. By the saddle-point inequality,
\begin{equation*}
p^\star \le L_0(x_{k+1},z_{k+1},\lambda^\star)=p_{k+1}+\langle \lambda^\star,r_{k+1}\rangle.
\end{equation*}

Since $r_{k+1}\to 0$, this gives
\begin{equation}
\label{eq:p_pstar_lower}
\liminf p_{k+1}\ge p^\star.
\end{equation}

For the upper bound, the optimality condition of \eqref{eq:matrix-x-update} reads $0\in \partial f(x_{k+1})+A^\top(\lambda_k+R_k e_{k+1})$. Using again \eqref{eq:matrix-lam-update}, this becomes $0\in \partial f(x_{k+1})+A^\top\lambda_{k+1}+A^\top\mu_{k+1}$, where
\begin{equation*}
\mu_{k+1}:=R_k\bigl((I-\Gamma_k)e_{k+1}-d_k\bigr)\to 0.
\end{equation*}

The optimality condition of \eqref{eq:matrix-z-update} gives $0\in \partial g(z_{k+1})+B^\top\lambda_{k+1}$. By convexity of $f$ and $g$,
\begin{align*}
f(x_{k+1})-f(x^\star) &\le -\langle \lambda_{k+1}+\mu_{k+1},\,Ax_{k+1}-Ax^\star\rangle,\\
g(z_{k+1})-g(z^\star) &\le -\langle \lambda_{k+1},\,Bz_{k+1}-Bz^\star\rangle.
\end{align*}

Adding and using $Ax^\star+Bz^\star=c$ yields
\begin{equation}
\label{eq:p_pstar}
p_{k+1}-p^\star \le -\langle \lambda_{k+1},r_{k+1}\rangle - \langle \mu_{k+1},Ax_{k+1}-Ax^\star\rangle.
\end{equation}

It remains to show that the right-hand side vanishes asymptotically. From the proof of Lemma~\ref{lem:quasi-fejer}, the sequence $\{\sigma_k\}$ is bounded. Since $\lambda_k=y_k-R_k\sigma_k$ and $\{R_k\}$ is uniformly bounded, $\{\lambda_k\}$ is bounded. Also, $Ax_{k+1}=\sigma_k+e_{k+1}$, so $\{Ax_{k+1}\}$ is bounded because $\{\sigma_k\}$ is bounded and $e_{k+1}\to 0$. Therefore, using $r_{k+1}\to 0$ and $\mu_{k+1}\to 0$ in \eqref{eq:p_pstar}, we obtain $\limsup p_{k+1}\le p^\star$. Together with \eqref{eq:p_pstar_lower}, this proves that $p_{k+1}\to p^\star$. 
\end{proof}


\section{Learning over-relaxation for OSQP}\label{sec:method}

The result of Theorem \ref{thm:main} shows that we can let a nonlinear function approximator, such as a neural network, choose the relaxation parameters online to improve algorithm performance on a distribution of problems of interest. If such updates comply with the summable-change condition, convergence for the worst-case realization of problems within this distribution is preserved.

In this section, we therefore turn to \luca{designing time-varying per-constraint over-relaxation parameters with the goal of enhancing the performance of } OSQP \cite{stellato2020osqp} -- a well-established algorithm that solves convex quadratic
programs by reformulating them as a consensus problem and applying ADMM.  Specifically, the OSQP iteration is an instance of the matrix-valued
relaxed ADMM~\eqref{eq:matrix-x-update}--\eqref{eq:matrix-lam-update}
in which $f$ encodes the quadratic objective together with a
linear-equality constraint via an indicator function,
$g$ is the indicator of a convex constraint set~$\mathcal{C}$,
$A=I$, and $B=-I$.
The penalty matrix~$R_k$ is diagonal with entries determined by a
scalar parameter~$\rho_k$ that OSQP updates only when a heuristic
triggering condition is met. \luca{The relaxation matrix $\Gamma_k$ is instead not chosen adaptively, and it is set equal to $1.6 I$. }

\luca{Since changing the penalty matrix $R_k$ triggers a costly matrix refactorization, in this paper we  focus on designing the relaxation matrix $ \Gamma_k $. Updating $\Gamma_k$ at every iteration does not require additional matrix factorizations and can therefore be performed online at negligible additional cost}.  \luca{Inspired by L2O approaches \cite{chen2022learning}, we propose learning the sequence of over-relaxation parameters by learning a \emph{policy} that, at each iteration $k$, maps a set of observed features to a chosen over-relaxation matrix $\Gamma_k$. }

Specifically, we propose a supervised-learning procedure, where for each quadratic program in the example set the primal-dual optimal pair $(x^\star,\lambda^\star)$ is known and used for training. At the current iteration $k \in \mathbb{N}$,  \luca{a neural-network policy determines a relaxation matrix $\Gamma_k$ based on a set of features, whose content we detail later. This choice of $\Gamma_k$ is kept fixed for the next $t$ iterations of ADMM, when a new relaxation matrix $\Gamma_{k+t}$ is chosen by the policy. In order to evaluate the performance of our algorithm for a given choice of policy parameters, and then backpropagate through such parameters, we define the loss incurred by the algorithm at iteration $k$ as} 

\luca{
\begin{equation}
\ell_k
=
\psi\!\left(
\log
\sqrt{
\frac{
\|x_{k+t}-x^\star\|_2^2 + \|\lambda_{k+t}-\lambda^\star\|_2^2 + \varepsilon
}{
\|x_k-x^\star\|_2^2 + \|\lambda_k-\lambda^\star\|_2^2 + \varepsilon
}
}
\right),
\label{eq:stage_loss}
\end{equation}
The relaxation policy is then designed to choose a sequence of $\Gamma_k$ that minimizes the cumulative loss over a finite-horizon, on average over the set of example quadratic programs. In \eqref{eq:stage_loss}, the shaping function
\begin{equation}
\psi(r)=\mathrm{softplus}(r+0.5)-0.5
\end{equation}
behaves as $\psi(r) \approx -0.5$ when $r \ll 0$ and $\psi(r) \approx r$ when $r > 0$, thereby saturating the loss when convergence is already rapid and preserving full penalization when residuals increase. 
}


Each training batch is cold-started from $x=z=\lambda=0$, which is the same as the default strategy of OSQP. To stabilize learning, feature normalization statistics are estimated once from short baseline rollouts with the default OSQP relaxation parameter and then frozen for training and inference.


We consider the training of two variants of relaxation policies.  
The first variant, which we denote as \emph{scalar policy}, predicts a matrix  in the form $$\Gamma_k= \alpha_k(\phi_k,\theta)I\,.$$ 

Here, $\alpha_k:\mathbb{R}^d  \times \mathbb{R}^D\mapsto \mathbb{R}$ is a trainable policy that outputs a scalar relaxation parameter based on a set of learnable parameters $\theta \in \mathbb{R}^D$ and a vector of observed features $\phi_k \in \mathbb{R}^d$ defined as
\begin{align*}
\phi_{k} &= \Bigl(
\log\|r\|_\infty,\;
\log\|s\|_\infty,\; \notag\\
&\phantom{{}={}\bigl(}
\log\tfrac{\|r\|_\infty}{\|r_{-}\|_\infty+\varepsilon},\;
\log\tfrac{\|s\|_\infty}{\|s_{-}\|_\infty+\varepsilon},\;
\log\tfrac{\|r\|_\infty}{\|s\|_\infty+\varepsilon}
,\; \log \rho\Bigr)^\top\!.
\end{align*}

Above, $r$ and $s$  denote the current primal and dual residuals at iteration $k$,  $\rho$ denotes the penalty parameter at iteration $k$, while $r_{-}$ and $s_{-}$ denote the corresponding residuals $t$~iterations before, and $\epsilon>0$ is a small value.

The second variant, which we denote as \emph{vector policy}, predicts instead a general diagonal matrix in the form $$\Gamma_k = \operatorname{diag}(\vec{\alpha}_k(\phi_k,\phi'_k,\varphi))\,,$$ where $\vec{\alpha}_k:\mathbb{R}^d\times \mathbb{R}^{d_2}\times \mathbb{R}^{E}\mapsto \mathbb{R}^p$ is a vector-valued policy,  $\varphi\in \mathbb{R}^E$ is a set of trainable parameters, and $\phi'_k$ is an additional set of per-constraint features defined as
\begin{align*}
\phi'_k &= \Bigl(
\log(z_i{-}l_i),\; \log(u_i{-}z_i),\; \log|r_i|,\;
\mathrm{sign}(r_i),\; \notag\\
&\phantom{{}={}\bigl(}
\log|\lambda_i|, 
\log\tfrac{|r_i|}{|r_{i,-}|+\varepsilon},\;
\log\rho_i,\; \|A_i\|_\infty
\Bigr)^\top\!,
\end{align*}
where $l$ and $u$ denote the lower and upper bounds at iteration $k$, $z$ denotes the auxiliary
variable at iteration $k$, and $A_i$ denotes the $i$-th row of the constraint matrix. All logarithmic features are clamped to $[-6,6]$ before feature normalization.

Since OSQP's diagonal penalty entries take only two values---$\rho$ for inequality constraints and $10^3\rho$ for equality constraints---the global scalar $\rho$ is the natural summary feature for the scalar policy. In the vector policy, however, per-constraint $\log \rho_i$ is already encoded in $\phi'_k$, which makes $\log \rho$ in $\phi_k$ redundant and we therefore drop it.

Both variants of the relaxation policy use the same multilayer perceptron (MLP).  Each hidden layer is followed by Layer Normalization and use Exponential Linear Unit (ELU) as the activation function.
The linear output layer is connected to a sigmoid function to map the output in $[\alpha_{\min}, \alpha_{\max}]$.  
We impose
$(\alpha_{\min}+\alpha_{\max})/2 = 1.6$ and initialize the output layer near
zero so that the initial prediction is close to $\alpha=1.6$, matching the
default OSQP relaxation parameter and providing a stable starting point for
training. By sharing the weights across rows, the vector policy is row-equivariant, and therefore  both policies are compatible with varying numbers of constraints and different problem sizes.

\section{Experiments}\label{sec:experiments}

We evaluated the proposed adaptive over-relaxation scheme in (i)~benchmark quadratic program families from the OSQP benchmark suite, where we trained on small problem instances and assessed generalization to larger ones, and (ii)~a Model Predictive Control (MPC) benchmark, with fixed parameters and changing initial states. In both settings, the key practical advantage is that updating  $\Gamma_k$ at every iteration is refactorization-free, so the learned policy can react to the solver state online at negligible cost. The end-to-end computational overhead then depends only on  evaluating the output of the neural-network relaxation policy, which is lightweight for the scalar variant and modest for the vector variant. The implementation is available at
\url{https://github.com/JALin2000/learning_over_relaxation_admm}.


\subsection{Experimental setup}

In our experiments, we trained the proposed over-relaxation policies in Multi Layer Perceptron (MLP) form for ADMM under several configurations:
\begin{enumerate}
\item \emph{Scalar policy with adaptive $\rho$}: the proposed trained scalar policy with OSQP's adaptive-$\rho$ heuristic.
\item \emph{Vector policy with adaptive $\rho$}: the proposed trained vector policy with OSQP's adaptive-$\rho$ heuristic.
\item \emph{Scalar policy without adaptive $\rho$}: the proposed trained scalar policy with fixed $\rho = 0.1$.
\item \emph{Vector policy without adaptive $\rho$}: the proposed trained vector policy with fixed $\rho = 0.1$.
\end{enumerate}
We benchmarked these configurations against standard OSQP, where the default relaxation matrix \luca{$\Gamma_k =  1.6 I$} is used, considering both cases in which $\rho$ is chosen adaptively and is kept fixed as $\rho=0.1$. \textcolor{black}{Additionally, we benchmarked against ARADMM \cite{xu2017aradmm}, an adaptive relaxed ADMM method that jointly updates the scalar penalty $\rho$ and relaxation parameters $\Gamma_k = \alpha I$ using spectral curvature estimates and correlation-based safeguarding. Following recommendations in ARADMM, we updated them for every 2 iterations, which is quite frequent compared with OSQP.}


We trained a separate policy for each problem family. The scalar and \luca{vector} policies share the same MLP architecture described in Section~\ref{sec:method}, with two hidden layers of 64 neurons each, and differ only in input dimension. Both predict $\Gamma_k$ in $[\alpha_{\min}, \alpha_{\max}] = [1.25, 1.95]$. We trained for 1000 epochs using differentiable unrolled OSQP with stage length $t = 10$. For the benchmark-suite experiments we used batch size 16, while for the MPC experiment we used batch size 10. We used AdamW \cite{loshchilov2017decoupledadamw} with learning rate $5 \times 10^{-5}$ and cosine decay. Feature normalization statistics were computed once from short baseline rollouts with $\Gamma_k=1.6I$ for all $k$ and held fixed thereafter. Training was carried out on the University of Oxford Advanced Research Computing (ARC) facility \cite{richards_university_2015}, and benchmarking was performed on a MacBook Pro with an Apple M4 CPU and 16\,GB memory.

\begin{table*}[t]
\centering
\caption{Benchmark-suite results: mean iterations and runtime (seconds).}
\label{tab:main_results}
\vspace{-3pt}

\renewcommand{\arraystretch}{1.35}

\resizebox{\textwidth}{!}{%
\begin{tabular}{c | c | ccccc}
\hline

$\rho$ adaptation & Policy & Random QP & Portfolio & Lasso & SVM & Control \\

\hline

\multirow{3}{*}{\shortstack[c]{No $\rho$\\adaptation}}
& OSQP   & 321.53 (0.803) & 454.64 (1.981) & 203.39 (1.972) & 654.10 (6.715) & 1331.41 (7.350) \\ \cdashline{2-7}
& Scalar policy & \textbf{263.34} \textbf{(0.653)} & \textbf{372.95} \textbf{(1.721)} & 172.64 \textbf{(1.667)} & 547.07 \textbf{(6.000)} & 1096.20 \textbf{(6.043)} \\
& Vector policy & 269.56 (0.748) & 373.13 (1.886) & \textbf{172.44} (1.761) & \textbf{545.49} (6.332) & \textbf{1092.95} (6.196) \\
\hline

\multirow{6}{*}{\shortstack[c]{With $\rho$\\adaptation}}
& OSQP            & 321.53 (0.808) & 272.30 (1.645) & 175.58 (3.644) & 346.32 (7.959) & 127.48 \textbf{(1.518)} \\
& ARADMM & 507.77 (10.600) & 6381.86 (34.072) & \textbf{31.36} (35.751) & 699.08 \textbf{(6.918)} & 406.63 (3.224) \\ \cdashline{2-7}
& Scalar policy (iter)   & \textbf{262.92} \textbf{(0.662)} & \textbf{233.58} \textbf{(1.468)} & 175.62 (3.847) & \textbf{341.45} (7.949) & \textbf{126.86} (1.551) \\
& Scalar policy ($\rho$) & \textbf{262.92} (0.666) & 324.26 (1.858) & 179.96 (7.042) & 462.78 (7.059) & 261.26 (2.438) \\
& Vector policy (iter)   & 268.41 (0.755) & 251.16 (1.595) & 180.09 (3.654) & 343.53 (7.642) & 140.13 (1.657) \\
& Vector policy ($\rho$) & 265.65 (0.736) & 366.33 (2.038) & 174.76 \textbf{(1.772)} & 425.22 (7.836) & 218.13 (2.100) \\

\hline
\end{tabular}%
}

\vspace{6pt}
\raggedright
{\footnotesize
OSQP: unmodified baseline. Scalar policy: $\Gamma_k= \alpha_kI$. Vector policy: $\Gamma_k = \operatorname{diag}(\vec{\alpha}_k)$.
(iter) / ($\rho$): checkpoint selected by lowest mean iteration count / fewest $\rho$ updates.
Bold marks the best result per column within each $\rho$-adaptation group. 
}
\vspace{-3mm}
\end{table*}

For models trained with adaptive $\rho$, we saved two checkpoints based on validation-set performance: one selected by the lowest mean iteration count, and one selected by the lowest mean number of $\rho$ updates. This is useful because a method with fewer OSQP iterations does not necessarily achieve lower wall-clock time when $\Gamma_k$ updates trigger more $\rho$ updates and additional refactorizations.

In all experiments, we ran all policies with the same low-precision tolerances
\[
\epsilon_{\mathrm{abs}}=\epsilon_{\mathrm{rel}}=10^{-3},
\]
with the stopping-tolerance definitions following OSQP \cite{stellato2020osqp}. All problem instances were successfully solved to this prescribed accuracy.

The learned policy is queried once every $t=10$ OSQP iterations \textcolor{black}{to limit the computational overhead of neural network inference}, consistent with the interval used for training. To satisfy the summable-change requirement underlying the convergence result in Section~\ref{sec:convergence_theory}, we disabled further $\Gamma_k$ updates after 500 OSQP iterations and kept the last predicted value fixed for the remainder of the solve. \textcolor{black}{Note that freezing is only one way to enforce summable-change condition. Alternatively, the proposed change in $\Gamma_k$ can be damped by a decaying factor $q_k=e^{-ck}$, $c>0$. In practice, we also observed stable convergence without such a safeguard, although this case is not covered by Theorem \ref{thm:main}.}

For each test instance, we recorded both the total number of OSQP iterations and the wall-clock runtime, which includes policy inference and all solver overhead.

\subsection{Benchmark-suite experiments}


First, we evaluated the over-relaxation policy on parametric quadratic program families from the OSQP benchmark suite\footnote{\url{https://github.com/osqp/osqp_benchmarks}}. In this benchmark-suite study, we assessed transfer across problem dimensions within a fixed problem family. Accordingly, for this part only, training and validation used a fixed size for each family, while testing used larger instances from the same family.

We report results on: Random QP, Portfolio Optimization, Lasso, Support Vector Machine (SVM), and Control benchmark QPs. These problem families are representative of repeated parametric optimization settings and are also used in prior work on learned acceleration of OSQP \cite{stellato2020osqp, ichnowski2021acceleratingrlqp}.

For training and validation, we used a fixed size for each family: 250 for Random QP, 20 for Portfolio, Lasso, and SVM, and 100 for the Control benchmark family, with 160 instances for training and 80 instances for validation per family. For testing, we used larger instances: sizes (500, 501, 503, 507, 515, 531, 562, 625, 750, 999) for Random QP, sizes (200, 201, 203, 205, 210, 219, 235, 262, 311, 399) for the Control benchmark family, and sizes (50, 51, 52, 54, 58, 63, 72, 87, 110, 149) for Portfolio, Lasso, and SVM, with 10 test instances per size and 100 test instances per problem family.


Table~\ref{tab:main_results} reports mean OSQP iterations and wall-clock runtime (in seconds) across 100 test instances per family for the benchmark-suite setting.

With fixed $\rho$, the over-relaxation policies consistently outperform baseline OSQP across all five families in both iteration count and wall-clock time, with iteration reductions of 15--18\%. Because $\Gamma_k$ updates are refactorization-free, the iteration savings translate directly into runtime gains. The scalar policy achieves the best runtime on every family; the vector policy occasionally yields marginally fewer iterations (e.g.\ Lasso, SVM, Control), but its per-row inference overhead can offset this advantage in wall-clock time.

\begin{table}[t]
\centering
\caption{MPC results: Mean iterations and runtime (seconds).}
\label{tab:mpc_varying_x0}
\renewcommand{\arraystretch}{1.35}
\resizebox{0.47\textwidth}{!}{%
\begin{tabular}{c|c|c}
\hline
$\rho$ adaptation & Policy & MPC with varying $x_0$ \\
\hline
\multirow{3}{*}{\shortstack[c]{No $\rho$\\adaptation}}
& OSQP    & 890.7 (0.745) \\ \cdashline{2-3}
& Scalar \luca{policy} & 733.4 \textbf{(0.621)} \\
& Vector \luca{policy} & \textbf{730.9} (0.691) \\
\hline
\multirow{6}{*}{\shortstack[c]{With $\rho$\\adaptation}}
& OSQP             & 141.4 (0.169) \\ 
& ARADMM & 355.4 (0.380) \\
\cdashline{2-3}
& Scalar \luca{policy} (iter)    & 133.5 \textbf{(0.168)} \\
& Scalar \luca{policy} ($\rho$)  & 139.7 (0.176) \\
& Vector \luca{policy} (iter)    & 126.5 (0.174) \\
& Vector \luca{policy} ($\rho$)  & \textbf{125.8} (0.183) \\
\hline
\end{tabular}
}
\vspace{-5mm}
\end{table}

With adaptive $\rho$, the over-relaxation policies still reduce iteration counts on most families---most notably Random~QP and Portfolio. Runtime gains are more nuanced: changes in the $\Gamma_k$ trajectory can alter the heuristic $\rho$ rule, triggering additional refactorizations. 
Different checkpoint strategies show this trade-off: the iteration-selected checkpoint (iter) minimizes iterations but may incur more $\rho$ updates, while the $\rho$-selected checkpoint ($\rho$) sacrifices iteration count for fewer refactorizations. A striking example is Lasso, where the vector policy with the $\rho$-selected checkpoint  halves the runtime relative to baseline OSQP (1.772\,s vs.\ 3.644\,s) at a comparable iteration count, demonstrating that a learned over-relaxation policy can implicitly regularize the $\rho$-update frequency.

\textcolor{black}{ARADMM \cite{xu2017aradmm} provides a useful comparison with the proposed policies as a strong baseline. On Lasso, ARADMM requires only 31.36 iterations, but this reduction does not translate into wall-clock time because tens of $\rho$-updates, and hence refactorizations, are performed for each problem instance. On SVM, ARADMM achieves the lowest runtime among the adaptive-\(\rho\) methods, but remains slower than all fixed-\(\rho\) variants. On other problem families, ARADMM is slower than baseline OSQP and the proposed policies in both iterations and runtime. Results show that frequent adaptation of $\rho$ yields substantial iteration reductions on some problem families, but such gains do not consistently translate across problems or into wall-clock improvements. Further tuning of ARADMM's hyperparameters may improve its performance, but identifying effective settings across problem families without learning can be challenging.}

In summary, across all five families, learning over-relaxation policies reliably reduces iteration counts compared to standard OSQP. The clearest runtime gains arise when $\rho$ is fixed, since $\Gamma_k$ adaptation is entirely refactorization-free. With adaptive $\rho$, iteration improvements persist, and careful checkpoint selection can convert them into runtime gains by controlling the interaction with OSQP's $\rho$-update heuristic.

\subsection{MPC with varying initial states}

Next, we evaluated the learned policies on an MPC experiment that more directly reflects the repeated-solve setting motivating this work. In MPC, the same system dynamics, cost matrices, and constraints are fixed across solves and thus can be learned from examples, while only the initial state $x_0$ varies. Accordingly, all instances in this experiment share the same randomly generated linear dynamics $A$ and $B$ with state dimension $n_x=100$ and input dimension $n_u=50$, quadratic cost matrices $Q$, $R$, $Q_T$, horizon length $T=10$, and state/input box constraints $x_{\min}, x_{\max}, u_{\min}, u_{\max}$. Only the initial state $x_0$ varies from instance to instance. 
We used 160 training instances, 80 validation instances, and 100 test instances generated from disjoint random seeds.
Table~\ref{tab:mpc_varying_x0} reports the corresponding results.

With fixed $\rho$, over-relaxation policies reduce the mean iteration count by 18\% and mean runtime by 17\%. As in the benchmark suite, the refactorization-free nature of $\Gamma_k$ updates ensures that iteration savings translate directly into wall-clock gains.

With adaptive $\rho$, the baseline is already strong, yet over-relaxation policies still reduce iterations by up to 11\%. Runtime differences are small in this regime. The scalar policy matches the baseline wall-clock time while the
vector policy trades a slight runtime increase for a larger iteration
reduction, consistent with the per-row inference overhead observed
in the benchmark-suite experiments. \textcolor{black}{By contrast, ARADMM performs worse than other methods in both iteration count and wall-clock time on this MPC benchmark. Together with the benchmark-suite results, this indicates that its benefit is problem dependent.}

To conclude, the over-relaxation policies are effective in the repeated-solve MPC regime as well.
When $\rho$ is fixed, the learned policies yield substantial gains in
both iterations and runtime. Under adaptive $\rho$, they consistently
improve iteration counts, with wall-clock differences that are small
and policy-dependent.

\section{Conclusion}


We studied how to accelerate ADMM for parametric convex programs by learning to adapt the relaxation parameter $\alpha$ online. The main insight is that $\alpha$ is an especially attractive adaptation target: in OSQP-like implementations, updating $\alpha$ does not trigger the matrix refactorizations required by penalty updates, so it can be adapted frequently at negligible solver-side cost. Our convergence result showed that such time-varying per-constraint adaptation of both $\alpha$ and $\rho$ is compatible with asymptotic convergence guarantees, and our experiments confirmed that the learned policies consistently reduce iteration counts, with the clearest wall-clock gains arising when refactorization overhead is eliminated entirely. Future work includes \textcolor{black}{interpreting the learned policies and reproducing their behavior through simpler and interpretable adaptation rules}, jointly learning coordinated adaptation policies for both $\rho$ and $\alpha$, benchmarking the framework beyond quadratic programs to broader classes of structured convex problems, and investigating applications in large-scale distributed optimization, where ADMM remains a central computational tool and low-overhead parameter adaptation could be especially valuable.

\section{ACKNOWLEDGMENTS}

The authors acknowledge the use of the University of Oxford Advanced Research Computing (ARC) facility \cite{richards_university_2015}.


\balance
\bibliographystyle{IEEEtran}
\bibliography{bibliography}






\end{document}